\documentclass[11pt]{article}
       \usepackage{amsfonts}
       \usepackage{stmaryrd}
       \usepackage{latexsym,amssymb,mathrsfs,fancyhdr}
       \font\tenmsb=msbm10
       \font\sevenmsb=msbm7
       \font\fivemsb=msbm5
       \catcode`\@=11
       \ifx\amstexloaded@\relax\catcode`\@=\active
       \endinput\else\let\amstexloaded@\relax\fi
       \def\spaces@{\space\space\space\space\space}
       \def\spaces@@{\spaces@\spaces@\spaces@\spaces@\spaces@}
       \def\space@.  {\futurelet\space@\relax}
       \space@.   %
       \def\Err@#1{\errhelp\defaulthelp@\errmessage{AmS-teX error: #1}}
       \def\relaxnext@{\let\next\relax}
       \def\accentfam@{7}
       \def\noaccents@{\def\accentfam@{0}}
       \def\Cal{\relaxnext@\ifmmode\let\next\Cal@\else
       \def\next{\Err@{Use \string\Cal\space only in math mode}}\fi\next}
       \def\Cal@#1{{\Cal@@{#1}}}
       \def\Cal@@#1{\noaccents@\fam\tw@#1}
       \def\Bbb{\relaxnext@\ifmmode\let\next\Bbb@\else
       \def\next{\Err@{Use \string\Bbb\space only in math mode}}\fi\next}
       \def\Bbb@#1{{\Bbb@@{#1}}}
       \def\Bbb@@#1{\noaccents@\fam\msbfam#1}
       \def\co{\tiny{\textcircled{\tiny\#}}}
       \newfam\msbfam
       \textfont\msbfam=\tenmsb
       \scriptfont\msbfam=\sevenmsb
       \scriptscriptfont\msbfam=\fivemsb
\usepackage{dsfont}
\usepackage[german,english]{babel}
\usepackage{amsmath,amssymb}
\usepackage[square, comma, sort&compress, numbers]{natbib}
\usepackage{cases}
\usepackage{mathrsfs}
\usepackage{amsmath}
\usepackage{amsthm}
\usepackage{amsfonts}
\usepackage{amssymb}
\usepackage{latexsym}
\usepackage{fancyhdr}
\usepackage{extarrows}

\newtheorem{thm}{Theorem}[section]

\newtheorem{lem}[thm]{Lemma}
\newtheorem{rem}[thm]{Remark}
\newtheorem{iteration lemma}[thm]{iteration Lemma}

\newtheorem*{acknowledgements*}{ACKNOWLEDGEMENtS}

\begin{document}

\setlength{\columnsep}{5pt}
\title{\bf Core and Dual Core Inverses of a Sum of Morphisms}
\author{Tingting  Li\footnote{ E-mail: littnanjing@163.com},
\ Jianlong Chen\footnote{ Corresponding author. E-mail: jlchen@seu.edu.cn},
\ Sanzhang Xu\footnote{ E-mail: xusanzhang5222@126.com}\\
Department of  Mathematics, Southeast University \\  Nanjing 210096,  China }
     \date{}

\maketitle
\begin{quote}
{\textbf{}\small
Let $\mathscr{C}$ be an additive category with an involution $\ast$.
Suppose that
$\varphi : X \rightarrow X$ is a morphism of $\mathscr{C}$ with core inverse $\varphi^{\co} : X \rightarrow X$
and $\eta : X \rightarrow X$ is a morphism of $\mathscr{C}$ such that $1_X+\varphi^{\co}\eta$ is invertible.
Let $\alpha=(1_X+\varphi^{\co}\eta)^{-1},$
$\beta=(1_X+\eta\varphi^{\co})^{-1},$
$\varepsilon=(1_X-\varphi\varphi^{\co})\eta\alpha(1_X-\varphi^{\co}\varphi),$
$\gamma=\alpha(1_X-\varphi^{\co}\varphi)\beta^{-1}\varphi\varphi^{\co}\beta,$
$\sigma=\alpha\varphi^{\co}\varphi\alpha^{-1}(1_X-\varphi\varphi^{\co})\beta,$
$\delta=\beta^{\ast}(\varphi^{\co})^{\ast}\eta^{\ast}(1_X-\varphi\varphi^{\co})\beta.$
Then $f=\varphi+\eta-\varepsilon$ has a core inverse if and only if $1_X-\gamma$, $1_X-\sigma$ and $1_X-\delta$ are invertible.
Moreover,
the expression of the core inverse of $f$ is presented.
Let $R$ be a unital $\ast$-ring and $J(R)$ its Jacobson radical, if $a\in R^{\co}$ with core inverse $a^{\co}$ and $j\in J(R)$,
then $a+j\in R^{\co}$ if and only if $(1-aa^{\co})j(1+a^{\co}j)^{-1}(1-a^{\co}a)=0$.
We also give the similar results for the dual core inverse.


}
\end{quote}

\section{ Introduction }\label{a}
Let $\mathscr{C}$ be an additive category with an involution $\ast$. (See, for example, \cite[p. 131]{PR2}.)
Let $\varphi : X \rightarrow Y$ and $\chi : Y \rightarrow X$ be morphisms of $\mathscr{C}$.
Consider the following four equations:
\begin{center}
  $(1)$ $\varphi\chi\varphi=\varphi$,~~~$(2)$ $\chi\varphi\chi=\chi$,~~~$(3)$ $(\varphi\chi)^{\ast}=\varphi\chi$,~~~$(4)$ $(\chi\varphi)^{\ast}=\chi\varphi$.
\end{center}
Let $\varphi\{i,j,\cdots,l\}$ denote the set of morphisms $\chi$ which satisfy equations $(i),(j),\cdots,(l)$ from among equations $(1)$-$(4)$.
If $\varphi\{i,j,\cdots,l\}\neq \emptyset$,
then $\varphi$ is called $\{i,j,\cdots,l\}$-invertible.
A morphism $\chi\in \varphi\{i,j,\cdots,l\}$ is called an $\{i,j,\cdots,l\}$-inverse of $\varphi$ and denoted by $\varphi^{(i,j,\cdots,l)}$.
A $\{1\}$-inverse is called a von Neumann regular inverse or inner inverse.
If a morphism $\chi\in \varphi\{1,2,3,4\}$,
then it is called the Moore-Penrose inverse of $\varphi$.
If such $\chi$ exists,
then it is unique and denoted by $\varphi^{\dagger}$.
If $X=Y$,
$\chi\in \varphi\{1,2\}$ and $\varphi\chi=\chi\varphi$,
then $\chi$ is called the group inverse of $\varphi$.
If such $\chi$ exists,
then it is unique and denoted by $\varphi^{\#}$.
A morphism $\varphi : X \rightarrow X$ is said to be Hermitian if $\varphi^{\ast}=\varphi$.

Recall that a unital ring $R$ is said to be a unital $\ast$-ring if it has an involution provided that there is an anti-isomorphism $\ast$ such that
$(a^{\ast})^{\ast}=a, (a+b)^{\ast}=a^{\ast}+b^{\ast}$ and $(ab)^{\ast}=b^{\ast}a^{\ast}$ for all $a, b\in R$.
In 2010,
Baksalary and Trenkler introduced the core and dual core inverses of a complex matrix in \cite{OM}.
Raki\'{c} et al. \cite{DSR} generalized core inverses of a complex matrix to the case of an element in a ring with an involution $\ast$.
An element $x\in R$ is said to be a core inverse of $a$ if it satisfies
\begin{equation*}
\begin{split}
    axa=a,~xR=aR,~Rx=Ra^{\ast},
\end{split}
\end{equation*}
such an element $x$ is unique if it exists and denoted by $a^{\co}$.
And they also showed that $a^{\co}$ exists if and only if there exists $x\in R$ such that
\begin{equation*}
\begin{split}
    axa=a,~xax=x,~(ax)^{\ast}=ax,~ax^{2}=x,~xa^{2}=a.
\end{split}
\end{equation*}
In this case,
$x=a^{\co}$.
There is a dual concept of core inverses which is called dual core inverses.
The symbols $R^{-1}$, $R^{\dagger}$, $R^{\#}$, $R^{\{i,j,\cdots,l\}}$, $R^{\co}$ and $R_{\co}$ denote the set of all the invertible, Moore-Penrose invertible, group invertible, $\{i,j,\cdots,l\}$-invertible, core invertible and dual core invertible elements in $R$, respectively.

Hence one can define the notion of the core and dual core inverse in an additive category.
Let $\mathscr{C}$ be an additive category with an involution and $\varphi : X \rightarrow X$ a morphism of $\mathscr{C}$.
If there is a morphism $\chi : X \rightarrow X$ satisfying
\begin{center}
  $\varphi\chi\varphi=\varphi$,~$\chi\varphi\chi=\chi$,~$(\varphi\chi)^{\ast}=\varphi\chi$,~$\varphi\chi^2=\chi$,~$\chi\varphi^2=\varphi$,
\end{center}
then $\varphi$ is core invertible and $\chi$ is called the core inverse of $\varphi$.
If such $\chi$ exists,
then it is unique and denoted by $\varphi^{\co}$.
And the dual core inverse can be given dually.

Group inverses and Moore-Penrose inverses of morphisms were investigated some years ago.
(See, \cite{PR1}-\cite{PP}.)
In \cite{HP},
D. Huylebrouck and R. Puystjens gave a necessary and sufficient condition for the von Neumann regularity,
Moore-Penrose invertibility and group invertibility of $a+j$ in a ring $R$ with identity,
where $a$ is a von Neumann regular element of $R$ and $j$ is an element of the Jacobson radical of $R$.
In \cite{H},
D. Huylebrouck generalized these results to an additive category $\mathscr{C}$ under some sufficient conditions.
In \cite{You},
H. You and J.L. Chen proved these sufficient conditions were also necessary which completed Huylebrouck's results.
In this paper,
we give the necessary and sufficient conditions for the existence of core inverses and dual core inverses for $f=\varphi+\eta-\varepsilon$.
The core and dual core invertibility of $a+j$ is also considered in this paper,
where $a$ is a core invertible element of $R$ and $j$ is an element of the Jacobson radical of $R$.

Before investigate the core inverse of a sum of morphisms,
some auxiliary results should be presented.

\begin{lem} \cite[p. $201$]{Hartwig}\label{000}
Let $a\in R$,
we have the following results:\\
$(1)$ $a$ is $\{1,3\}$-invertible with $\{1,3\}$-inverse $x$ if and only if $x^{\ast}a^{\ast}a=a;$\\
$(2)$ $a$ is $\{1,4\}$-invertible with $\{1,4\}$-inverse $y$ if and only if $aa^{\ast}y^{\ast}=a.$
\end{lem}

\begin{lem} \cite[Proposition $7$]{Hartwig}\label{111}
Let $a\in R$,
$a\in R^{\#}$ if and only if $a=a^{2}x=ya^{2}$ for some $x, y\in R$.
In this case, $a^{\#}=yax=y^{2}a=ax^{2}$.
\end{lem}

\begin{lem} \cite[Theorem $2.6$ and $2.8$]{XSZ}\label{222}
Let $a\in R$,
we have the following results:\\
$(1)$ $a\in R^{\co}$ if and only if $a\in R^{\#}\cap R^{\{1,3\}}$. In this case, $a^{\co}=a^{\#}aa^{(1,3)}$.\\
$(2)$ $a\in R_{\co}$ if and only if $a\in R^{\#}\cap R^{\{1,4\}}$. In this case, $a_{\co}=a^{(1,4)}aa^{\#}$.
\end{lem}

\begin{lem} \cite[Theorem~$3.4$]{Li}\label{n=1-core-inverse}
Let $a\in R$.
The following conditions are equivalent: \\
$(1)$ $a\in R^{\co}$; \\
$(2)$ there exists a unique projection $p$ such that $pa=0$, $u=a+p\in R^{-1}$;\\
$(3)$ there exists a Hermitian element $p$ such that $pa=0$, $u=a+p\in R^{-1}$.\\
In this case,
\begin{equation*}
\begin{split}
    a^{\co}=u^{-1}au^{-1}=(u^{\ast}u)^{-1}a^{\ast}.
\end{split}
\end{equation*}
\end{lem}

It should be pointed out,
the above lemmas are valid in an additive category with an involution $\ast$.

\begin{lem} \cite[Proposition~$1$]{H}\label{(12)-inverse}
Let $\mathscr{C}$ be an additive category.
Suppose that
$\varphi : X \rightarrow Y$ is a morphism of $\mathscr{C}$ with $\{1,2\}-$inverse $\varphi^{(1,2)}$
and $\eta : X \rightarrow Y$ is a morphism of $\mathscr{C}$ such that $1_X+\varphi^{(1,2)}\eta$ is invertible.
Let
\begin{equation*}
\begin{split}
    \varepsilon=(1_Y-\varphi\varphi^{(1,2)})\eta(1_X+\varphi^{(1,2)}\eta)^{-1}(1_X-\varphi^{(1,2)}\varphi),
\end{split}
\end{equation*}
then $f=\varphi+\eta-\varepsilon$ has a $\{1\}$-inverse
and $(1_X+\varphi^{(1,2)}\eta)^{-1}\varphi^{(1,2)}\in f\{1,2\}$.
Moreover,
if $\tau\in (\varphi+\eta)\{1\}$,
then we have $\tau\in \varepsilon\{1\}$.
\end{lem}

\section{Core and Dual Core Inverses of a Sum of Morphisms}\label{a}
Let $\mathscr{C}$ be an additive category with an involution $\ast$.
Suppose that
both $\varphi : X \rightarrow Y$ and $\eta : X \rightarrow Y$ are morphisms of $\mathscr{C}$.
We use the following notations:
\begin{equation*}
\begin{split}
    \alpha &~=(1_X+\varphi^{\tau}\eta)^{-1},\\
    \beta &~=(1_Y+\eta\varphi^{\tau})^{-1},\\
    \varepsilon &~=(1_Y-\varphi\varphi^{\tau})\eta\alpha(1_X-\varphi^{\tau}\varphi),\\
    \gamma &~=\alpha(1_X-\varphi^{\tau}\varphi)\eta\varphi^{\tau}\beta,\\
    \delta &~=\alpha\varphi^{\tau}\eta(1_Y-\varphi\varphi^{\tau})\beta,\\
    \lambda &~=\alpha(1_X-\varphi^{\tau}\varphi)\eta^{\ast}(\varphi^{\tau})^{\ast}\alpha^{\ast},\\
    \mu &~=\beta^{\ast}(\varphi^{\tau})^{\ast}\eta^{\ast}(1_Y-\varphi\varphi^{\tau})\beta,
\end{split}
\end{equation*}
where $\tau \in \{\#, (1,2,3), (1,2,4), \dagger\}$.

In $2001$,
You and Chen \cite{You} gave the  group inverse,
$\{1,2,4\}$-inverse,
$\{1,2,3\}$-inverse and Moore-Penrose inverse of a sum of morphisms in an additive category,
respectively.
The results are as follows.\\
(I) \cite[Proposition~$1$]{You}\label{group inverse}~Let $X=Y$ and $\tau=\#$.
If $\varphi$ is group invertible with group inverse $\varphi^{\#} : X \rightarrow X$ and $1_X+\varphi^{\#}\eta$ is invertible,
then the following conditions are equivalent:\\
(i) $f=\varphi+\eta-\varepsilon$ has a group inverse;\\
(ii) $1_X-\gamma$ and $1_X-\delta$ are invertible;\\
(iii) $1_X-\gamma$ is left invertible and $1_X-\delta$ is right invertible.\\
In this case,
\begin{equation*}
\begin{split}
    f^{\#}~=(1_X-\gamma)^{-1}\alpha\varphi^{\#}(1_X-\delta)^{-1},
\end{split}
\end{equation*}
\begin{equation*}
\begin{split}
    (1_X-\gamma)^{-1} &~=1_X-\varphi^{\#}\varphi+f^{\#}f\varphi^{\#}\varphi,\\
    (1_X-\delta)^{-1} &~=1_X-\varphi\varphi^{\#}+\varphi\varphi^{\#}ff^{\#}.
\end{split}
\end{equation*}
(II) \cite[Proposition~$2$]{You}\label{(124) inverse}~Let $\tau=(1,2,4)$.
If $\varphi$ is $\{1,2,4\}$-invertible with $\{1,2,4\}$-inverse $\varphi^{(1,2,4)} : Y \rightarrow X$ and $1_X+\varphi^{(1,2,4)}\eta$ is invertible,
then the following conditions are equivalent:\\
(i) $f=\varphi+\eta-\varepsilon$ is $\{1,2,4\}$-invertible;\\
(ii) $1_X-\lambda$ is invertible;\\
(iii) $1_X-\lambda$ is left invertible.\\
In this case,
\begin{equation*}
\begin{split}
    (1_X-\lambda)^{-1}\alpha\varphi^{(1,2,4)}\in f\{1,2,4\},
\end{split}
\end{equation*}
\begin{equation*}
\begin{split}
    (1_X-\lambda)^{-1}~=1_X-\varphi^{(1,2,4)}\varphi+f^{(1,2,4)}f\varphi^{(1,2,4)}\varphi.
\end{split}
\end{equation*}
(III) \cite[Proposition~$3$]{You}\label{(123) inverse}~Let $\tau=(1,2,3)$.
If $\varphi$ is $\{1,2,3\}$-invertible with $\{1,2,3\}$-inverse $\varphi^{(1,2,3)} : Y \rightarrow X$ and $1_X+\varphi^{(1,2,3)}\eta$ is invertible,
then the following conditions are equivalent:\\
(i) $f=\varphi+\eta-\varepsilon$ is $\{1,2,3\}$-invertible;\\
(ii) $1_Y-\mu$ is invertible;\\
(iii) $1_Y-\mu$ is right invertible.\\
In this case,
\begin{equation*}
\begin{split}
    \varphi^{(1,2,3)}\beta(1_Y-\mu)^{-1}\in f\{1,2,3\},
\end{split}
\end{equation*}
\begin{equation*}
\begin{split}
    (1_Y-\mu)^{-1}~=1_Y-\varphi\varphi^{(1,2,3)}+\varphi\varphi^{(1,2,3)}ff^{(1,2,3)}.
\end{split}
\end{equation*}
(IV) \cite[Proposition~$4$]{You}\label{Group inverse}~Let $\tau=\dagger$.
If $\varphi$ is Moore-Penrose invertible with Moore-Penrose inverse $\varphi^{\dagger} : Y \rightarrow X$ and $1_X+\varphi^{\dagger}\eta$ is invertible,
then the following conditions are equivalent:\\
(i) $f=\varphi+\eta-\varepsilon$ has an Moore-Penrose inverse;\\
(ii) $1_X-\lambda$ and $1_Y-\mu$ are invertible;\\
(iii) $1_X-\lambda$ is left invertible and $1_Y-\mu$ is right invertible.\\
In this case,
\begin{equation*}
\begin{split}
    f^{\dagger}~=(1_X-\lambda)^{-1}\alpha\varphi^{\dagger}(1_Y-\mu)^{-1},
\end{split}
\end{equation*}
\begin{equation*}
\begin{split}
    (1_X-\lambda)^{-1} &~=1_X-\varphi^{\dagger}\varphi+f^{\dagger}f\varphi^{\dagger}\varphi,\\
    (1_X-\mu)^{-1} &~=1_Y-\varphi\varphi^{\dagger}+\varphi\varphi^{\dagger}ff^{\dagger}.
\end{split}
\end{equation*}

Inspired by the above results,
we investigate similar results for core inverses and dual core inverses in an additive category with an involution $\ast$.
It should be pointed out,
the above notations are no longer used below.

\begin{thm} \label{core-inverse category}
Let $\mathscr{C}$ be an additive category with an involution $\ast$.
Suppose that
$\varphi : X \rightarrow X$ is a morphism of $\mathscr{C}$ with core inverse $\varphi^{\co}$
and $\eta : X \rightarrow X$ is a morphism of $\mathscr{C}$ such that $1_X+\varphi^{\co}\eta$ is invertible.
Let
\begin{equation*}
\begin{split}
    \alpha &~=(1_X+\varphi^{\co}\eta)^{-1},\\
    \beta &~=(1_X+\eta\varphi^{\co})^{-1},\\
    \varepsilon &~=(1_X-\varphi\varphi^{\co})\eta\alpha(1_X-\varphi^{\co}\varphi),\\
    \gamma &~=\alpha(1_X-\varphi^{\co}\varphi)\beta^{-1}\varphi\varphi^{\co}\beta,\\
    \sigma &~=\alpha\varphi^{\co}\varphi\alpha^{-1}(1_X-\varphi\varphi^{\co})\beta,\\
    \delta &~=\beta^{\ast}(\varphi^{\co})^{\ast}\eta^{\ast}(1_X-\varphi\varphi^{\co})\beta.
\end{split}
\end{equation*}
Then the following conditions are equivalent:\\
(i) $f=\varphi+\eta-\varepsilon$ has a core inverse;\\
(ii) $1_X-\gamma$, $1_X-\sigma$ and $1_X-\delta$ are invertible;\\
(iii) $1_X-\gamma$ is left invertible, both $1_X-\sigma$ and $1_X-\delta$ are right invertible.\\
In this case,
\begin{equation*}
\begin{split}
    f^{\co}~=(1_X-\gamma)^{-1}\alpha\varphi^{\co}(1_X-\delta)^{-1},
 \end{split}
\end{equation*}
\begin{equation*}
\begin{split}
    (1_X-\gamma)^{-1} &~=1_X-\varphi\varphi^{\co}+f^{\co}f\varphi\varphi^{\co},\\
    (1_X-\sigma)^{-1} &~=1_X-\varphi\varphi^{\co}+\varphi\varphi^{\co}f^{\co}f,\\
    (1_X-\delta)^{-1} &~=1_X-\varphi\varphi^{\co}+\varphi\varphi^{\co}ff^{\co}.
\end{split}
\end{equation*}
\end{thm}

\begin{proof}
While the method of this proof is similar to You and Chen's (see \cite{You}),
there is still enough different about them.

By Lemma~\ref{(12)-inverse},
$(1_X+\varphi^{\co}\eta)^{-1}\varphi^{\co}\in f\{1,2\}$.

Let $f_0=(1_X+\varphi^{\co}\eta)^{-1}\varphi^{\co}=\alpha\varphi^{\co}=\varphi^{\co}\beta,$
then
\begin{equation*}
\begin{split}
    \varphi^{\co}f=\varphi^{\co}(\varphi+\eta-\varepsilon)=\varphi^{\co}\varphi+\varphi^{\co}\eta=\varphi^{\co}\varphi(1_X+\varphi^{\co}\eta)=\varphi^{\co}\varphi\alpha^{-1},
\end{split}
\end{equation*}
and $f\varphi^{\co}=\beta^{-1}\varphi\varphi^{\co}.$
So $f_0f=\alpha\varphi^{\co}f=\alpha\varphi^{\co}\varphi\alpha^{-1}$,
and
\begin{equation*}
\begin{split}
    1_X-f_0f
    &~=1_X-\alpha\varphi^{\co}\varphi\alpha^{-1}~=\alpha(1_X-\varphi^{\co}\varphi)\alpha^{-1}\\
    &~=\alpha(1_X-\varphi^{\co}\varphi)(1_X+\varphi^{\co}\eta)~=\alpha(1_X-\varphi^{\co}\varphi).
\end{split}
\end{equation*}
Similarly,
we have $ff_0=\beta^{-1}\varphi\varphi^{\co}\beta$ and $1_X-ff_0~=(1_X-\varphi\varphi^{\co})\beta.$\\
Further,
\begin{equation*}
\begin{split}
    (1_X-f_0f)ff_0=\alpha(1_X-\varphi^{\co}\varphi)\beta^{-1}\varphi\varphi^{\co}\beta=\gamma,
\end{split}
\end{equation*}
\begin{equation*}
\begin{split}
    f_0f(1_X-ff_0)=\alpha\varphi^{\co}\varphi\alpha^{-1}(1_X-\varphi\varphi^{\co})\beta=\sigma,
\end{split}
\end{equation*}
\begin{equation*}
\begin{split}
    (ff_0)^{\ast}(1_X-ff_0)
    &~=(\beta^{-1}\varphi\varphi^{\co}\beta)^{\ast}(1_X-\varphi\varphi^{\co})\beta\\
    &~=\beta^{\ast}\varphi\varphi^{\co}(\beta^{-1})^{\ast}(1_X-\varphi\varphi^{\co})\beta\\
    &~=\beta^{\ast}[(1_X-\varphi\varphi^{\co})\beta^{-1}\varphi\varphi^{\co}]^{\ast}\beta\\
    &~=\beta^{\ast}[(1_X-\varphi\varphi^{\co})(1_X+\eta\varphi^{\co})\varphi\varphi^{\co}]^{\ast}\beta\\
    &~=\beta^{\ast}[(1_X-\varphi\varphi^{\co})(\varphi\varphi^{\co}+\eta\varphi^{\co})]^{\ast}\beta\\
    &~=\beta^{\ast}[(1_X-\varphi\varphi^{\co})\eta\varphi^{\co}]^{\ast}\beta\\
    &~=\delta.
\end{split}
\end{equation*}
Therefore,
we obtain
\begin{equation*}
\begin{split}
    f\gamma=0,~~~\sigma f=0,~~~\delta f=0,
\end{split}
\end{equation*}
moreover,
\begin{equation*}
\begin{split}
    f_0f^2
    &~=f-f+f_0f^2=f-(1_X-f_0f)f\\
    &~=[1_X-(1_X-f_0f)ff_0]f=(1_X-\gamma)f,
\end{split}
\end{equation*}
\begin{equation*}
\begin{split}
    f^2f_0
    &~=f-f+f^2f_0=f-f(1_X-ff_0)\\
    &~=f[1_X-f_0f(1_X-ff_0)]=f(1_X-\sigma),
\end{split}
\end{equation*}
\begin{equation*}
\begin{split}
    f^{\ast}ff_0
    &~=f^{\ast}-f^{\ast}+f^{\ast}ff_0=f^{\ast}-f^{\ast}(1_X-ff_0)\\
    &~=f^{\ast}[1_X-(ff_0)^{\ast}(1_X-ff_0)]=f^{\ast}(1_X-\delta).
\end{split}
\end{equation*}

Now we are ready to show the equivalence of three conditions.

$(i)\Rightarrow (ii)$.
The first step is to show that $1_X-\varphi\varphi^{\co}+f^{\co}f\varphi\varphi^{\co}$ is the inverse of $1_X-\gamma$.\\
Note that
\begin{equation*}
\begin{split}
    (1_X-\gamma)f^{\co}f
    &~=(1_X-\gamma)ff^{\co}f^{\co}f=f_0f^{2}f^{\co}f^{\co}f=f_0f\\
    &~=\alpha\varphi^{\co}\varphi\alpha^{-1}=\alpha\varphi^{\co}\varphi(1_X+\varphi^{\co}\eta)=\alpha(\varphi^{\co}\varphi+\varphi^{\co}\eta)\\
    &~=\alpha(1_X+\varphi^{\co}\eta+\varphi^{\co}\varphi-1_X)=\alpha(\alpha^{-1}+\varphi^{\co}\varphi-1_X)\\
    &~=1_X+\alpha(\varphi^{\co}\varphi-1_X).
\end{split}
\end{equation*}
Post-multiplication $\varphi\varphi^{\co}$ on the equality above yields
\begin{equation*}
\begin{split}
    (1_X-\gamma)f^{\co}f\varphi\varphi^{\co}=\varphi\varphi^{\co}.
\end{split}
\end{equation*}
As
\begin{equation*}
\begin{split}
    \gamma(1_X-\varphi\varphi^{\co})
    &~=\alpha(1_X-\varphi^{\co}\varphi)\beta^{-1}\varphi\varphi^{\co}\beta(1_X-\varphi\varphi^{\co})\\
    &~=\alpha(1_X-\varphi^{\co}\varphi)\beta^{-1}\varphi\alpha\varphi^{\co}(1_X-\varphi\varphi^{\co})\\
    &~=0,
\end{split}
\end{equation*}
we obtain
\begin{equation*}
\begin{split}
    &~ ~~~(1_X-\gamma)(1_X-\varphi\varphi^{\co}+f^{\co}f\varphi\varphi^{\co})\\
    &~=1_X-\varphi\varphi^{\co}-\gamma(1_X-\varphi\varphi^{\co})+(1_X-\gamma)f^{\co}f\varphi\varphi^{\co}\\
    &~=1_X-\varphi\varphi^{\co}-0+\varphi\varphi^{\co}\\
    &~=1_X.
\end{split}
\end{equation*}
So $1_X-\varphi\varphi^{\co}+f^{\co}f\varphi\varphi^{\co}$ is the right inverse of $1_X-\gamma$.
Next,
we prove that $1_X-\varphi\varphi^{\co}+f^{\co}f\varphi\varphi^{\co}$ is also the left inverse of $1_X-\gamma$.\\
Note that
\begin{equation*}
\begin{split}
    \gamma
    &~=\alpha(1_X-\varphi^{\co}\varphi)\beta^{-1}\varphi\varphi^{\co}\beta\\
    &~=\alpha(1_X-\varphi^{\co}\varphi)(1_X+\eta\varphi^{\co})\varphi\varphi^{\co}\beta\\
    &~=\alpha(1_X-\varphi^{\co}\varphi)(\varphi\varphi^{\co}+\eta\varphi^{\co})\beta\\
    &~=\alpha(1_X-\varphi^{\co}\varphi)\eta\varphi^{\co}\beta,
\end{split}
\end{equation*}
$1_X-\varphi^{\co}\eta\alpha=\alpha$ and $1_X-\eta\varphi^{\co}\beta=\beta$,
thus we have
\begin{equation*}
\begin{split}
    \varphi\varphi^{\co}\gamma
    &~=\varphi\varphi^{\co}\alpha(1_X-\varphi^{\co}\varphi)\eta\varphi^{\co}\beta\\
    &~=\varphi\varphi^{\co}(1_X-\varphi^{\co}\eta\alpha)(1_X-\varphi^{\co}\varphi)\eta\varphi^{\co}\beta\\
    &~=\varphi\varphi^{\co}(1_X-\varphi^{\co}\varphi)\eta\varphi^{\co}\beta-\varphi^{\co}\eta\alpha(1_X-\varphi^{\co}\varphi)\eta\varphi^{\co}\beta\\
    &~=\varphi\varphi^{\co}(1_X-\varphi^{\co}\varphi)\eta\varphi^{\co}\beta+(\alpha-1_X)(1_X-\varphi^{\co}\varphi)\eta\varphi^{\co}\beta\\
    &~=\varphi\varphi^{\co}(1_X-\varphi^{\co}\varphi)\eta\varphi^{\co}\beta+\alpha(1_X-\varphi^{\co}\varphi)\eta\varphi^{\co}\beta-(1_X-\varphi^{\co}\varphi)\eta\varphi^{\co}\beta\\
    &~=\alpha(1_X-\varphi^{\co}\varphi)\eta\varphi^{\co}\beta-(1_X-\varphi\varphi^{\co})(1_X-\varphi^{\co}\varphi)\eta\varphi^{\co}\beta\\
    &~=\gamma-(1_X-\varphi\varphi^{\co})\eta\varphi^{\co}\beta\\
    &~=\gamma-(1_X-\varphi\varphi^{\co})(1_X-\beta).
\end{split}
\end{equation*}
So $(1_X-\varphi\varphi^{\co})\gamma=(1_X-\varphi\varphi^{\co})(1_X-\beta)$,
which implies
\begin{equation*}
\begin{split}
    (1_X-\varphi\varphi^{\co})(1_X-\gamma)=(1_X-\varphi\varphi^{\co})\beta.
\end{split}
\end{equation*}
Furthermore,
\begin{equation*}
\begin{split}
    f^{\co}f\varphi\varphi^{\co}\gamma
    &~=f^{\co}f[\gamma-(1_X-\varphi\varphi^{\co})(1_X-\beta)]\\
    &~=f^{\co}f\gamma-f^{\co}f(1_X-\varphi\varphi^{\co})(1_X-\beta),\\
    &~=-f^{\co}f(1_X-\varphi\varphi^{\co})(1_X-\beta).
\end{split}
\end{equation*}
In addition,
\begin{equation*}
\begin{split}
    &~~~~~1_X-\varphi\varphi^{\co}+f^{\co}f\varphi\varphi^{\co}\\
    &~=1_X+\eta\varphi^{\co}-f\varphi^{\co}+f^{\co}f\varphi\varphi^{\co}\\
    &~=1_X+\eta\varphi^{\co}-f^{\co}f(f\varphi^{\co}-\varphi\varphi^{\co})\\
    &~=1_X+\eta\varphi^{\co}-f^{\co}f\eta\varphi^{\co}\\
    &~=f^{\co}f+(1_X-f^{\co}f)(1_X+\eta\varphi^{\co})\\
    &~=f^{\co}f+(1_X-f^{\co}f)\beta^{-1}.
\end{split}
\end{equation*}
Therefore,
\begin{equation*}
\begin{split}
    &~~~~~(1_X-\varphi\varphi^{\co}+f^{\co}f\varphi\varphi^{\co})(1_X-\gamma)\\
    &~=(1_X-\varphi\varphi^{\co})(1_X-\gamma)+f^{\co}f\varphi\varphi^{\co}-f^{\co}f\varphi\varphi^{\co}\gamma\\
    &~=(1_X-\varphi\varphi^{\co})\beta+f^{\co}f\varphi\varphi^{\co}+f^{\co}f(1_X-\varphi\varphi^{\co})(1_X-\beta)\\
    &~=(1_X-\varphi\varphi^{\co})\beta+f^{\co}f\varphi\varphi^{\co}+f^{\co}f(1_X-\beta)-f^{\co}f\varphi\varphi^{\co}+f^{\co}f\varphi\varphi^{\co}\beta\\
    &~=f^{\co}f(1_X-\beta)+(1_X-\varphi\varphi^{\co}+f^{\co}f\varphi\varphi^{\co})\beta\\
    &~=f^{\co}f(1_X-\beta)+[f^{\co}f+(1_X-f^{\co}f)\beta^{-1}]\beta\\
    &~=f^{\co}f-f^{\co}f\beta+f^{\co}f\beta+1_X-f^{\co}f\\
    &~=1_X.
\end{split}
\end{equation*}
Hence $1_X-\gamma$ is invertible with inverse $(1_X-\gamma)^{-1}=1_X-\varphi\varphi^{\co}+f^{\co}f\varphi\varphi^{\co}.$

The second step is to prove that $1_X-\varphi\varphi^{\co}+\varphi\varphi^{\co}f^{\co}f$ is the inverse of $1_X-\sigma$.
On the one hand,
\begin{equation*}
\begin{split}
    f^{\co}f(1_X-\sigma)
    &~=f^{\co}f^2f_0=ff_0=\beta^{-1}\varphi\varphi^{\co}\beta=(1_X+\eta\varphi^{\co})\varphi\varphi^{\co}\beta\\
    &~=(\varphi\varphi^{\co}+\eta\varphi^{\co})\beta=(1_X+\eta\varphi^{\co}+\varphi\varphi^{\co}-1_X)\beta\\
    &~=1_X+(\varphi\varphi^{\co}-1_X)\beta.
\end{split}
\end{equation*}
Pre-multiplication $\varphi\varphi^{\co}$ on the equality above yields
\begin{equation*}
\begin{split}
    \varphi\varphi^{\co}f^{\co}f(1_X-\sigma)=\varphi\varphi^{\co}.
\end{split}
\end{equation*}
And because
\begin{equation*}
\begin{split}
    (1_X-\varphi\varphi^{\co})\sigma
    &~=(1_X-\varphi\varphi^{\co})\alpha\varphi^{\co}\varphi\alpha^{-1}(1_X-\varphi\varphi^{\co})\beta\\
    &~=(1_X-\varphi\varphi^{\co})\varphi^{\co}\beta\varphi\alpha^{-1}(1_X-\varphi\varphi^{\co})\beta\\
    &~=0,
\end{split}
\end{equation*}
we have
\begin{equation*}
\begin{split}
    &~~~~~(1_X-\varphi\varphi^{\co}+\varphi\varphi^{\co}f^{\co}f)(1_X-\sigma)\\
    &~=(1_X-\varphi\varphi^{\co})(1_X-\sigma)+\varphi\varphi^{\co}f^{\co}f(1_X-\sigma)\\
    &~=1_X-\varphi\varphi^{\co}+\varphi\varphi^{\co}\\
    &~=1_X.
\end{split}
\end{equation*}
On the other hand,
as $1_X-\beta\eta\varphi^{\co}=\beta$,
we obtain
\begin{equation*}
\begin{split}
    \sigma\varphi\varphi^{\co}
    &~=\alpha\varphi^{\co}\varphi\alpha^{-1}(1_X-\varphi\varphi^{\co})\beta\varphi\varphi^{\co}\\
    &~=\alpha\varphi^{\co}\varphi\alpha^{-1}(1_X-\varphi\varphi^{\co})(1_X-\beta\eta\varphi^{\co})\varphi\varphi^{\co}\\
    &~=\alpha\varphi^{\co}\varphi\alpha^{-1}(1_X-\varphi\varphi^{\co})\varphi\varphi^{\co}-\alpha\varphi^{\co}\varphi\alpha^{-1}(1_X-\varphi\varphi^{\co})\beta\eta\varphi^{\co}\varphi\varphi^{\co}\\
    &~=-\alpha\varphi^{\co}\varphi\alpha^{-1}(1_X-\varphi\varphi^{\co})\beta\eta\varphi^{\co}\\
    &~=-\alpha\varphi^{\co}\varphi\alpha^{-1}(1_X-\varphi\varphi^{\co})(1_X-\beta)\\
    &~=\sigma-\alpha\varphi^{\co}\varphi\alpha^{-1}(1_X-\varphi\varphi^{\co}).
\end{split}
\end{equation*}
Thus we have
\begin{equation*}
\begin{split}
    \sigma(1_X-\varphi\varphi^{\co})=\alpha\varphi^{\co}\varphi\alpha^{-1}(1_X-\varphi\varphi^{\co}),
\end{split}
\end{equation*}
and
\begin{equation*}
\begin{split}
    \sigma\varphi\varphi^{\co}f^{\co}f
    &~=[\sigma-\alpha\varphi^{\co}\varphi\alpha^{-1}(1_X-\varphi\varphi^{\co})]f^{\co}f\\
    &~=\sigma f^{\co}f-\alpha\varphi^{\co}\varphi\alpha^{-1}(1_X-\varphi\varphi^{\co})f^{\co}f\\
    &~=\sigma ff^{\co}f^{\co}f-\alpha\varphi^{\co}\varphi\alpha^{-1}(1_X-\varphi\varphi^{\co})f^{\co}f\\
    &~=-\alpha\varphi^{\co}\varphi\alpha^{-1}(1_X-\varphi\varphi^{\co})f^{\co}f.
\end{split}
\end{equation*}
In addition,
\begin{equation*}
\begin{split}
    &~~~~\varphi\varphi^{\co}+\alpha\varphi^{\co}\varphi\alpha^{-1}(1_X-\varphi\varphi^{\co})\\
    &~=\varphi\varphi^{\co}+\alpha\varphi^{\co}\varphi(1_X+\varphi^{\co}\eta)(1_X-\varphi\varphi^{\co})\\
    &~=\alpha[\alpha^{-1}\varphi\varphi^{\co}+\varphi^{\co}\varphi(1_X+\varphi^{\co}\eta)(1_X-\varphi\varphi^{\co})]\\
    &~=\alpha[(1_X+\varphi^{\co}\eta)\varphi\varphi^{\co}+(\varphi^{\co}\varphi+\varphi^{\co}\eta)(1_X-\varphi\varphi^{\co})]\\
    &~=\alpha[(1_X+\varphi^{\co}\eta)\varphi\varphi^{\co}+\varphi^{\co}\varphi+\varphi^{\co}\eta-(\varphi^{\co}\varphi+\varphi^{\co}\eta)\varphi\varphi^{\co}]\\
    &~=\alpha[\varphi^{\co}\varphi+\varphi^{\co}\eta+(1_X+\varphi^{\co}\eta-\varphi^{\co}\varphi-\varphi^{\co}\eta)\varphi\varphi^{\co}]\\
    &~=\alpha(\varphi^{\co}\varphi+\varphi^{\co}\eta)\\
    &~=\alpha\varphi^{\co}f.
\end{split}
\end{equation*}
Therefore
\begin{equation*}
\begin{split}
    &~~~~(1_X-\sigma)(1_X-\varphi\varphi^{\co}+\varphi\varphi^{\co}f^{\co}f)\\
    &~=(1_X-\sigma)(1_X-\varphi\varphi^{\co})+(1_X-\sigma)\varphi\varphi^{\co}f^{\co}f\\
    &~=1_X-\varphi\varphi^{\co}-\sigma(1_X-\varphi\varphi^{\co})+\varphi\varphi^{\co}f^{\co}f-\sigma\varphi\varphi^{\co}f^{\co}f\\
    &~=1_X-\varphi\varphi^{\co}-\alpha\varphi^{\co}\varphi\alpha^{-1}(1_X-\varphi\varphi^{\co})+\varphi\varphi^{\co}f^{\co}f+\alpha\varphi^{\co}\varphi\alpha^{-1}(1_X-\varphi\varphi^{\co})f^{\co}f\\
    &~=1_X-\varphi\varphi^{\co}+\varphi\varphi^{\co}f^{\co}f-\alpha\varphi^{\co}\varphi\alpha^{-1}(1_X-\varphi\varphi^{\co})(1_X-f^{\co}f)\\
    &~=1_X-\varphi\varphi^{\co}(1_X-f^{\co}f)-\alpha\varphi^{\co}\varphi\alpha^{-1}(1_X-\varphi\varphi^{\co})(1_X-f^{\co}f)\\
    &~=1_X-[\varphi\varphi^{\co}+\alpha\varphi^{\co}\varphi\alpha^{-1}(1_X-\varphi\varphi^{\co})](1_X-f^{\co}f)\\
    &~=1_X-\alpha\varphi^{\co}f(1_X-f^{\co}f)\\
    &~=1_X.
\end{split}
\end{equation*}
Thus $1_X-\sigma$ is invertible with inverse $(1_X-\sigma)^{-1}=1_X-\varphi\varphi^{\co}+\varphi\varphi^{\co}f^{\co}f.$

In the end,
$1_X-\delta$ is invertible with inverse $(1_X-\delta)^{-1}=1_X-\varphi\varphi^{\co}+\varphi\varphi^{\co}ff^{\co}$ can be deduced immediately by \cite[Proposition~$3$]{You}\label{(123) inverse}
and the fact that the core inverse is a $\{1,2,3\}$-inverse.

$(ii)\Rightarrow (iii)$. Obviously.

$(iii)\Rightarrow (i)$.
Assume that $\omega$ is the left inverse of $1_X-\gamma$.
$\nu$ and $\lambda$ are the right inverses of $1_X-\delta$ and $1_X-\sigma$,
respectively.
Then we have
\begin{equation}\label{01}
\begin{split}
    f=\omega(1_X-\gamma)f=\omega f_0f^2,
\end{split}
\end{equation}
\begin{equation}\label{02}
\begin{split}
    f=f(1_X-\sigma)\lambda=f^2f_0\lambda,
\end{split}
\end{equation}
and
\begin{equation}\label{03}
\begin{split}
    f=(f^{\ast})^{\ast}=(f^{\ast}(1_X-\delta)\nu)^{\ast}=(f^{\ast}ff_0\nu)^{\ast}=\nu^{\ast}f_0^{\ast}f^{\ast}f.
\end{split}
\end{equation}
By Lemma~\ref{111},
equalities (\ref{01}) and (\ref{02}) tell us that $f$ is group invertible with group inverse $f^{\#}=\omega f_0ff_0\lambda=\omega f_0\lambda$.
And by Lemma~\ref{000},
equality (\ref{03}) shows that $f$ is $\{1,3\}$-invertible with $f_0v\in f\{1,3\}$.
Hence,
$f$ is core invertible by Lemma~\ref{222}.
Moreover,
\begin{equation*}
\begin{split}
    f^{\co}=f^{\#}ff^{(1,3)}=\omega f_0\lambda ff_0\nu.
\end{split}
\end{equation*}

From $(i)\Rightarrow (ii)$,
we know that $1_X-\gamma$,
$1_X-\sigma$ and $1_X-\delta$ are invertible with inverse $\omega=(1_X-\gamma)^{-1}$,
$\lambda=(1_X-\sigma)^{-1}$ and $\nu=(1_X-\delta)^{-1}$,
respectively.
As $\sigma f=0$,
we have $(1_X-\sigma)f=f$,
that is to say $f=(1_X-\sigma)^{-1}f=\lambda f.$
Therefore,
\begin{equation*}
\begin{split}
    f^{\co}=\omega f_0\lambda ff_0\nu=\omega f_0ff_0\nu=\omega f_0\nu=(1_X-\gamma)^{-1}\alpha\varphi^{\co}(1_X-\delta)^{-1}.
\end{split}
\end{equation*}
\end{proof}

\begin{rem}
Conditions as in Theorem~\ref{core-inverse category},
if $1_X-\delta$ is invertible,
then $ff_0$ is core invertible with $(ff_0)^{\co}=ff_0(1_X-\delta)^{-1}$.
\begin{proof}
Since $1_X-ff_0=(1_X-\varphi\varphi^{\co})\beta$,
\begin{equation*}
\begin{split}
    1_X-\delta-ff_0
    &~=(1_X-ff_0)-\delta\\
    &~=(1_X-\varphi\varphi^{\co})\beta-\beta^{\ast}(\varphi^{\co})^{\ast}\eta^{\ast}(1_X-\varphi\varphi^{\co})\beta\\
    &~=(1_X-\eta\varphi^{\co}\beta)^{\ast}(1_X-\varphi\varphi^{\co})\beta\\
    &~=\beta^{\ast}(1_X-\varphi\varphi^{\co})\beta.
\end{split}
\end{equation*}
Set $q=\beta^{\ast}(1_X-\varphi\varphi^{\co})\beta=1_X-\delta-ff_0$,
then $q=q^{\ast}$, $qff_0=0$ and $q+ff_0=1_X-\delta$ is invertible.
Thus $ff_0$ is core invertible and $(ff_0)^{\co}=(1_X-\delta)^{-1}ff_0(1_X-\delta)^{-1}$ by Lemma~\ref{n=1-core-inverse}.
Since $\delta f=0$,
$(1_X-\delta)f=f$,
which implies $f=(1_X-\delta)^{-1}f$.
Hence
\begin{equation*}
\begin{split}
    (ff_0)^{\co}=(1_X-\delta)^{-1}ff_0(1_X-\delta)^{-1}=ff_0(1_X-\delta)^{-1}.
 \end{split}
\end{equation*}
\end{proof}
\end{rem}

There is a result for the dual core inverse,
which corresponds to Theorem~\ref{core-inverse category},
as follows.

\begin{thm} \label{dual-core-inverse category}
Let $\mathscr{C}$ be an additive category with an involution $\ast$.
Suppose that
$\varphi : X \rightarrow X$ is a morphism of $\mathscr{C}$ with dual core inverse $\varphi_{\co} : X \rightarrow X$
and $\eta : X \rightarrow X$ is a morphism of $\mathscr{C}$ such that $1_X+\varphi_{\co}\eta$ is invertible.
Let
\begin{equation*}
\begin{split}
    \alpha &~=(1_X+\varphi_{\co}\eta)^{-1},\\
    \beta &~=(1_X+\eta\varphi_{\co})^{-1},\\
    \varepsilon &~=(1_X-\varphi\varphi_{\co})\eta\alpha(1_X-\varphi_{\co}\varphi),\\
    \rho &~=\alpha\varphi_{\co}\varphi\alpha^{-1}(1_X-\varphi\varphi_{\co})\beta,\\
    \zeta &~=\alpha(1_X-\varphi_{\co}\varphi)\beta^{-1}\varphi\varphi_{\co}\beta,\\
    \xi &~=\alpha(1_X-\varphi_{\co}\varphi)\eta^{\ast}(\varphi_{\co})^{\ast}\alpha^{\ast}.
\end{split}
\end{equation*}
Then the following conditions are equivalent:\\
(i) $f=\varphi+\eta-\varepsilon$ has a dual core inverse;\\
(ii) $1_X-\rho$, $1_X-\zeta$ and $1_X-\xi$ are invertible;\\
(iii) $1_X-\rho$ is right invertible, both $1_X-\zeta$ and $1_X-\xi$ are left invertible.\\
In this case,
\begin{equation*}
\begin{split}
    f_{\co}~=(1_X-\xi)^{-1}\alpha\varphi_{\co}(1_X-\rho)^{-1},
 \end{split}
\end{equation*}
\begin{equation*}
\begin{split}
    (1_X-\rho)^{-1} &~=1_X-\varphi_{\co}\varphi+\varphi_{\co}\varphi ff_{\co},\\
    (1_X-\zeta)^{-1} &~=1_X-\varphi_{\co}\varphi+ff_{\co}\varphi_{\co}\varphi,\\
    (1_X-\xi)^{-1} &~=1_X-\varphi_{\co}\varphi+f_{\co}f\varphi_{\co}\varphi.
\end{split}
\end{equation*}
\end{thm}

\section{Core and Dual Core Inverses of a Sum with a radical element}\label{b}
Let $R$ be a unital $\ast$-ring and $J(R)$ its Jacobson radical.
As a matter of convenience,
we use the following notation:
\begin{center}
  $\varepsilon_\tau=(1-aa^\tau)j(1+a^\tau j)^{-1}(1-a^\tau a)$,
\end{center}
where $\tau \in \{(1), (1,2,3), (1,2,4), \dagger, \#\}$ and $j\in J(R)$.

In \cite{HP},
Huylebrouck and Puystjens proved the following results.\\
(I)~If $a\in R^{\{1\}}$, then $a+j\in R^{\{1\}}$ if and only if $\varepsilon_{(1)}=0$.\\
(II)~If $a\in R^{\dagger}$, then $a+j\in R^{\dagger}$ if and only if $\varepsilon_{\dagger}=0$.\\
(III)~If $a\in R^{\#}$, then $a+j\in R^{\#}$ if and only if $\varepsilon_{\#}=0$.

In \cite{You},
You and Chen told us:\\
(IV)~If $a\in R^{\{1,2,3\}}$, then $a+j\in R^{\{1,2,3\}}$ if and only if $\varepsilon_{(1,2,3)}=0$.\\
(V)~If $a\in R^{\{1,2,4\}}$, then $a+j\in R^{\{1,2,4\}}$ if and only if $\varepsilon_{(1,2,4)}=0$.\\
Moreover,
the expressions of $(a+j)^{(1,2)}$,
$(a+j)^{\dagger}$,
$(a+j)^{\#}$,
$(a+j)^{(1,2,3)}$,
$(a+j)^{(1,2,4)}$ are presented, respectively.

Next,
we will show that the core invertible element has a similar result as the above.

\begin{thm} \label{core-inverse}
Let $R$ be a unital $\ast$-ring and $J(R)$ its Jacobson radical.
If $a\in R^{\co}$ with core inverse $a^{\co}$ and $j\in J(R)$,
then
\begin{center}
  $a+j\in R^{\co}$ if and only if $\varepsilon=(1-aa^{\co})j(1+a^{\co}j)^{-1}(1-a^{\co}a)=0$.
\end{center}
In this case,
\begin{equation*}
\begin{split}
    (a+j)^{\co}=(1-\gamma)^{-1}(1+a^{\co}j)^{-1}a^{\co}(1-\delta)^{-1},
\end{split}
\end{equation*}
where
\begin{equation*}
\begin{split}
    \gamma &~=(1+a^{\co}j)^{-1}(1-a^{\co}a)(1+ja^{\co})aa^{\co}(1+ja^{\co})^{-1},\\
    \delta &~=(1+(a^{\co})^{\ast}j^{\ast})^{-1}(a^{\co})^{\ast}j^{\ast}(1-aa^{\co})(1+ja^{\co})^{-1}.
\end{split}
\end{equation*}
\end{thm}

\begin{proof}
Remark first that,
if $j\in J(R)$,
then $1+a^{\co}j\in R^{-1}$ and $j^{\ast}\in J(R)$.

Set $\phi=(a+j)^{\co}$,
then $\phi\in \varepsilon\{1\}$ by Lemma~\ref{(12)-inverse}.
This shows that the element $\varepsilon\in J(R)$ is von Neumann regular,
so it must be zero, that is to say $\varepsilon=(1-aa^{\co})j(1+a^{\co}j)^{-1}(1-a^{\co}a)=0$.

On the contrary,
suppose $\varepsilon=(1-aa^{\co})j(1+a^{\co}j)^{-1}(1-a^{\co}a)=0$,
then it is easy to see that $(1+a^{\co}j)^{-1}a^{\co}\in (a+j)\{1,2\}$ by Lemma~\ref{(12)-inverse} and the fact that the core inverse is a $\{1,2\}$-inverse.
Thus,
we have
\begin{equation*}
\begin{split}
    (a+j)(1+a^{\co}j)^{-1}a^{\co}(a+j)=(a+j),
\end{split}
\end{equation*}
which implies
\begin{equation}\label{04}
\begin{split}
    [1-(a+j)(1+a^{\co}j)^{-1}a^{\co}](a+j)=0,
\end{split}
\end{equation}
where
\begin{equation*}
\begin{split}
    &~~~~~1-(a+j)(1+a^{\co}j)^{-1}a^{\co}\\
    &~=1-(a+j)a^{\co}(1+ja^{\co})^{-1}\\
    &~=(1+ja^{\co})(1+ja^{\co})^{-1}-(a+j)a^{\co}(1+ja^{\co})^{-1}\\
    &~=[(1+ja^{\co})-(a+j)a^{\co}](1+ja^{\co})^{-1}\\
    &~=(1-aa^{\co})(1+ja^{\co})^{-1}.
\end{split}
\end{equation*}
Hence the equality (\ref{04}) can be written as
\begin{equation}\label{05}
\begin{split}
    (1-aa^{\co})(1+ja^{\co})^{-1}(a+j)=0.
\end{split}
\end{equation}
Since $a\in R^{\co}$,
set $p=1-aa^{\co}$,
then $p$ is a Hermitian element such that $pa=0$, $a+p\in R^{-1}$ by the proof of Lemma~\ref{n=1-core-inverse}.
Let
\begin{equation*}
\begin{split}
    q=[(1+ja^{\co})^{-1}]^{\ast}p(1+ja^{\co})^{-1}=[(1+ja^{\co})^{-1}]^{\ast}(1-aa^{\co})(1+ja^{\co})^{-1},
\end{split}
\end{equation*}
then $q=q^{\ast}$ and
\begin{equation*}
\begin{split}
      q(a+j)=[(1+ja^{\co})^{-1}]^{\ast}(1-aa^{\co})(1+ja^{\co})^{-1}(a+j)~\stackrel{(\ref{05})}{=}0.
\end{split}
\end{equation*}
Moreover,
we have
\begin{equation*}
\begin{split}
    a+q
    &~=a+[(1+ja^{\co})^{-1}]^{\ast}p(1+ja^{\co})^{-1}\\
    &~=[(1+ja^{\co})^{-1}]^{\ast}[(1+ja^{\co})^{\ast}a(1+ja^{\co})+p](1+ja^{\co})^{-1},
\end{split}
\end{equation*}
where
\begin{equation*}
\begin{split}
    &~ ~~~(1+ja^{\co})^{\ast}a(1+ja^{\co})+p\\
    &~=a+aja^{\co}+(a^{\co})^{\ast}j^{\ast}a+(a^{\co})^{\ast}j^{\ast}aja^{\co}+p\\
    &~=(a+p)+aja^{\co}+(a^{\co})^{\ast}j^{\ast}a+(a^{\co})^{\ast}j^{\ast}aja^{\co}\\
    &~=(a+p)[1+(a+p)^{-1}(aja^{\co}+(a^{\co})^{\ast}j^{\ast}a+(a^{\co})^{\ast}j^{\ast}aja^{\co})]
\end{split}
\end{equation*}
is invertible which follows from the property of Jacobson radical and the fact that $a+p\in R^{-1}$.
Thus $a+q\in R^{-1}$.
Therefore,
\begin{equation*}
\begin{split}
    a+j+q=(a+q)[1+(a+q)^{-1}j]\in R^{-1}.
\end{split}
\end{equation*}
In conclusion,
$q$ is a Hermitian element such that $q(a+j)=0$, $a+j+q\in R^{-1}$.
Applying Lemma~\ref{n=1-core-inverse},
we can get that $a+j\in R^{\co}$.
Therefore,
$1_X-\gamma$, $1_X-\sigma$ and $1_X-\delta$ are invertible by Theorem~\ref{core-inverse category},
where
\begin{equation*}
\begin{split}
    \gamma &~=(1+a^{\co}j)^{-1}(1-a^{\co}a)(1+ja^{\co})aa^{\co}(1+ja^{\co})^{-1},\\
    \sigma &~=(1+a^{\co}j)^{-1}a^{\co}a(1+a^{\co}j)(1-aa^{\co})(1+ja^{\co})^{-1},\\
    \delta &~=(1+(a^{\co})^{\ast}j^{\ast})^{-1}(a^{\co})^{\ast}j^{\ast}(1-aa^{\co})(1+ja^{\co})^{-1}.
\end{split}
\end{equation*}
Hence we obtain
\begin{equation*}
\begin{split}
    (a+j)^{\co}=(1-\gamma)^{-1}(1+a^{\co}j)^{-1}a^{\co}(1-\delta)^{-1}.
\end{split}
\end{equation*}
\end{proof}

Similar to Theorem~\ref{core-inverse}, we have

\begin{thm} \label{dual-core-inverse}
Let $R$ be a unital $\ast$-ring and $J(R)$ its Jacobson radical.
If $a\in R_{\co}$ with dual core inverse $a_{\co}$ and $j\in J(R)$,
then
\begin{center}
  $a+j\in R_{\co}$ if and only if $\varepsilon=(1-aa_{\co})j(1+a_{\co}j)^{-1}(1-a_{\co}a)=0$.
\end{center}
In this case,
\begin{equation*}
\begin{split}
    (a+j)_{\co}=(1-\xi)^{-1}(1+a_{\co}j)^{-1}a_{\co}(1-\rho)^{-1},
\end{split}
\end{equation*}
where
\begin{equation*}
\begin{split}
    \rho &~=(1+a_{\co}j)^{-1}a_{\co}a(1+a_{\co}j)(1-aa_{\co})(1+ja_{\co})^{-1},\\
    \xi &~=(1+a_{\co}j)^{-1}(1-a_{\co}a)j^{\ast}(a_{\co})^{\ast}(1+j^{\ast}a_{\co}^{\ast})^{-1}.
\end{split}
\end{equation*}
\end{thm}

\vspace{0.2cm} \noindent {\large\bf Acknowledgements}

This research is supported by the National Natural Science Foundation of China (No.11201063 and No.11371089); the Natural Science Foundation of Jiangsu Province (No.BK20141327).

\end{document}